\documentclass[11pt]{article}

\usepackage{amsmath, amsfonts,amssymb, amsthm}
\usepackage{mathtools}
\usepackage{multirow}
\usepackage{graphicx}
\usepackage{color}
\usepackage[table]{xcolor}
\usepackage{epsf}
\usepackage{tikz}
\usepackage{subfig}
\usepackage{hyperref}
\hypersetup{
    colorlinks = true,
    linkcolor = {black},
    citecolor = {black}
}

\usepackage[left=1.3in, right=1.3in, top=1.2in, bottom=1.2in]{geometry}

\renewenvironment{thebibliography}[1]{
  \begin{oldthebibliography}{#1}
    \setlength{\itemsep}{.5em}
    \setlength{\parskip}{0em}
}
{
  \end{oldthebibliography}
}

\usepackage{enumitem}
\setlist{itemsep=0pt, topsep=0pt}

\newcommand{\floor}[1]{\lfloor#1\rfloor}
\newcommand{\ceiling}[1]{\lceil#1\rceil}

\newcommand{\RR}{\mathbb{R}}

\newcommand{\PP}{\mathbb{P}}
\newcommand{\EE}{\mathbb{E}}

\newcommand{\tbf}[1]{\textbf{#1}}

\newtheorem{theorem}{Theorem}[section]
\newtheorem{corollary}[theorem]{Corollary}
\newtheorem{lemma}[theorem]{Lemma}

\newtheorem{proposition}[theorem]{Proposition}
\newtheorem{observation}[theorem]{Observation}

\newtheorem{problem}[theorem]{Problem}

\newcommand{\sqbs}[1]{\left[ #1 \right]}
\newcommand{\of}[1]{\left( #1 \right)}

\title{A note on the multicolor size-Ramsey numbers of connected graphs}
\author{Louis DeBiasio\thanks{Department of Mathematics, Miami University, Oxford, OH. \texttt{debiasld@miamioh.edu}. Research supported in part by NSF grant DMS-1954170.}}
\date{\today}

\begin{document}

\maketitle

\begin{abstract}
The $r$-color size-Ramsey number of a graph $H$, denoted by $\widehat{R}_r(H)$, is the minimum number of edges in a graph $G$ having the property that every $r$-coloring of the edges of $G$ contains a monochromatic copy of $H$. 

Krivelevich \cite{K1} proved that $\widehat{R}_r(P_{m+1})=\Omega(r^2m)$ where $P_{m+1}$ is the path on $m$ edges.  He explains that his proof actually applies to any connected graph $H$ with $m$ edges and vertex cover number larger than $\sqrt{m}$.  He also notes that some restriction on the vertex cover number is necessary since the star with $m$ edges, $K_{1,m}$, has vertex cover number 1 and satisfies $\widehat{R}_r(K_{1,m})=r(m-1)+1$.  We prove that the star is actually the only exception; that is, $\widehat{R}_r(H)=\Omega(r^2m)$ for every non-star connected graph $H$ with $m$ edges.

We also prove a strengthening of this result for trees.  It follows from results of Beck \cite{B2} and Dellamonica \cite{Del} that $\widehat{R}_2(T)=\Theta(\beta(T))$ for every tree $T$ with bipartition $\{V_1, V_2\}$ and $\beta(T)=|V_1|\max\{d(v):v\in V_1\}+|V_2|\max\{d(v):v\in V_2\}$.  We prove that $\widehat{R}_r(T)=\Omega(r^2\beta(T))$ for every tree $T$, again with the exception of the star.  Additionally, we prove that for the family of non-star trees $T$ with $\beta(T)=\Omega(n_1n_2)$ (which includes all non-star trees of linear maximum degree and all trees of radius 2 for example) we have $\widehat{R}_r(T)=\Theta(r^2\beta(T))$.
\end{abstract}

\section{Introduction}

The $r$-color size-Ramsey number of a graph $H$, denoted by $\widehat{R}_r(H)$, is the minimum number of edges in a graph $G$ having the property that every $r$-coloring of the edges of $G$ contains a monochromatic copy of $H$.  When $r=2$ we drop the subscript.

In his study of the 2-color size-Ramsey number of trees, Beck introduced \cite{B2} the following parameter $\beta(\cdot)$ (and conjectured that the 2-color size-Ramsey number of every tree $T$ is essentially determined by $\beta(T)$).  First, we call $H$ an $(n_1, n_2,\Delta_1,\Delta_2)$-bipartite graph if $H$ is a connected bipartite graph with unique bipartition $\{V_1, V_2\}$ with $|V_i|=n_i$ and $\Delta_i=\max\{d(v):v\in V_i\}$ for all $i\in [2]$.  Given an $(n_1, n_2,\Delta_1,\Delta_2)$-bipartite graph $H$ let $$\beta(H)=n_1\Delta_1+n_2\Delta_2.$$

It is known that for every tree $T$, 
\begin{equation*}\label{eq:beck}
\frac{\beta(T)}{4}\leq \widehat{R}(T)=O(\beta(T))
\end{equation*}
where the lower bound is due to Beck \cite{B2}, and the upper bound is due to 
Dellamonica \cite{Del}. In fact, Dellamonica's result \cite{Del} actually implies that for all $r\geq 2$,
\begin{equation}\label{eq:del}
\widehat{R}_r(T)=O_r(\beta(T)).  
\end{equation}
For further discussion regarding the dependence of the hidden constant on $r$ in \eqref{eq:del}, see Section \ref{sec:con}.

While the 2-color size-Ramsey number of trees has been studied extensively (see \cite{FP}, \cite{HK}, \cite{Ke} in addition to the results mentioned above), much less is known about the $r$-color size-Ramsey number of trees (aside from the special cases of paths and stars).  First note that for stars with $m$ edges, it is trivial to see that $\widehat{R}_r(K_{1,m})=r(m-1)+1$.

For the path on $m$ edges $P_{m+1}$, Krivelevich \cite{K1} proved that for all $r\geq 2$, $\widehat{R}_r(P_{m+1})=O(r^2(\log r) m)$ (the constant was later improved by Dudek and Pra{\l}at \cite{DP2} to give $\widehat{R}_r(P_{m+1})< 600r^2(\log r) m$).  However, with just a few modifications, Krivelevich's proof can be generalized to show (see Appendix) that for all $r, \Delta\geq 2$ there exists $m_0$ such that if $T$ is a tree with $m\geq m_0$ edges and maximum degree at most $\Delta$, then $\widehat{R}_r(T)= O(\Delta^3 r^2(\log (r\Delta))m$.

Regarding lower bounds, Dudek and Pra{\l}at \cite{DP1} proved that for all $r\geq 2$, $\widehat{R}_r(P_{m+1})\geq \frac{r^2}{4}m$.  Soon after, Krivelevich \cite{K1} gave a different proof based on affine planes which showed that $\widehat{R}_r(P_{m+1})\geq (r-2-o_r(1))^2m$ (later, Bal and the author \cite{BD1} slightly refined Krivelevich's proof to show that $\widehat{R}_r(P_{m+1})\geq (r-1-o_r(1))^2m$ and also gave yet another different proof to show that for all $r\geq 2$, $\widehat{R}_r(P_{m+1})\geq \frac{(r-1)r}{2}m$).  A noteworthy aspect of Krivelevich's result is that he explains how his proof actually applies to any connected graph $H$ with $m$ edges and vertex cover number $\tau(H)$ significantly larger than $\sqrt{m}$.  He also notes that some restriction on the vertex cover number is necessary since $\tau(K_{1,m})=1$ and $\widehat{R}_r(K_{1,m})=r(m-1)+1$.

This raises two questions: for which connected graphs $H$ with $m$ edges is it true that $\widehat{R}_r(H)=\Omega(r^2m)$, and for which trees $T$ is it true that $\widehat{R}_r(T)=\Omega(r^2\beta(T))$?  We answer both questions, showing that in both cases stars are the only exception. In fact, we answer the second question more generally in terms of connected bipartite graphs.

\begin{theorem}\label{thm1}
For all integers $r\geq 2$ there exists $m_0:=m_0(r)$ such that if $H$ is a connected graph with $m\geq m_0$ edges and $H$ is not a star, then $\widehat{R}_r(H)> \frac{r^2}{72}m$.
\end{theorem}

\begin{theorem}\label{thm2}
For all integers $r\geq 6$, there exists $n_0:=n_0(r)$ such that if $H$ is a connected bipartite graph on $n\geq n_0$ vertices and $H$ is not a star, then $\widehat{R}_r(H)\geq \frac{r^2}{2304}\beta(H)$.
\end{theorem}

We remark that it should be possible to modify our proof of Theorem \ref{thm2} to remove the restriction that $r\geq 6$ at the expense of having an absolute constant smaller than $\frac{1}{2304}$, and (more importantly) an additional step in the proof to separately deal with the case when $2\leq r\leq 5$.  However, since we are only interested in the long term behavior in terms of $r$, we choose to keep the proof as simple as possible.  Also we make no serious attempt to optimize the values of $m_0$, $n_0$ or the absolute constants $\frac{1}{72}$, $\frac{1}{2304}$ appearing in the lower bounds, but we point out that $m_0$ and $n_0$ are polynomial in $r$.

The double star $S_{n,m}$ is the tree on $n+m+2$ vertices obtained by joining the centers of $K_{1,n}$ and $K_{1,m}$.  We determine the correct order of magnitude (in terms of $r$) of $\widehat{R}_r(S_{n,m})$.

\begin{theorem}\label{thm3} For all integers $r\geq 2$ and $n\geq m\geq 1$ such that $n+m+2\geq \frac{(\ceiling{r/2}+1)^2}{2}$, we have
\begin{equation*}
\begin{rcases}
  \frac{1}{2}(m+1)(n+m+2)  &r=2 \\
  \frac{r^2-1}{16}m(n+m+2)  &r\geq 3
\end{rcases} 
\leq \widehat{R}_r(S_{n,m})\leq 4r^2nm+2r(n+m)+1.
\end{equation*}
\end{theorem}

In the process of proving the above result, we realized that we were able to determine the correct order of magnitude (in terms of $r$) of the $r$-color size-Ramsey numbers of a much larger family of trees.  Note that given an $(n_1, n_2, \Delta_1, \Delta_2)$-bipartite graph $H$, the largest possible value of $\beta(H)=\Delta_1n_1+\Delta_2n_2$ is $2n_1n_2$ which is achieved by $S_{n_1-1, n_2-1}$ as well as $K_{n_1, n_2}$ for instance.  Given $0<\alpha\leq 2$, we say that a connected bipartite graph $H$ is \emph{$\alpha$-full} if $\beta(H)\geq \alpha n_1n_2$.  Note that all $(n_1, n_2, \Delta_1, \Delta_2)$-trees (or more generally $(n_1, n_2, \Delta_1, \Delta_2)$-bipartite graphs) of radius 2 -- which includes non-trivial double stars -- are $\alpha$-full for some $\alpha\geq 1$ (since there is necessarily a vertex in the part of size $n_i$ which is adjacent to every vertex in the part of size $n_{3-i}$ for some $i\in [2]$).  Also all $(n_1, n_2, \Delta_1, \Delta_2)$-trees (or more generally $(n_1, n_2, \Delta_1, \Delta_2)$-bipartite graphs) with maximum degree $\alpha(n_1+n_2)$ are $\alpha'$-full for some $\alpha'\geq \alpha$.

\begin{theorem}\label{thm4}
For all integers $r\geq 6$ there exists $n_0$ such that for all $0<\alpha\leq 2$, if $T$ is an $\alpha$-full $(n_1, n_2, \Delta_1, \Delta_2)$-tree with $n_1+n_2\geq n_0$ and $T$ is not a star, then $$\frac{\alpha r^2}{2304} n_1n_2 \leq \widehat{R}_r(T)\leq 4r^2 n_1n_2+4r(n_1+n_2)+1;$$ i.e.~$\widehat{R}_r(T)=\Theta(r^2\beta(T))$.
\end{theorem}

\section{Notation and Preliminary material}

\subsection{Notation}

Given a graph $G$ and a set $S\subseteq V(G)$, let $\Delta(S)=\max\{d(v): v\in S\}$ and let $G[S]$ be the subgraph induced by $S$.  We write $e(G)$ to mean $|E(G)|$.  Given disjoint non-empty sets $X,Y\subseteq V(G)$, we let $G[X,Y]$ be the bipartite graph on $X\cup Y$ induced by the edges between $X$ and $Y$.  For a subgraph $G'\subseteq G$, a set $S\subseteq V(G')$, and a vertex $v\in V(G')$, we write $N_{G'}(v)=\{u: \{u,v\}\in E(G')\}$, $N_{G'}(S)=\bigcup_{v\in S}N_{G'}(v)$, and $d_{G'}(v,S)=|N_{G'}(v)\cap S|$.  If $G'=G$, we drop the subscripts.  

The \emph{vertex cover number} of a graph $G$, denoted by $\tau(G)$, is the smallest positive integer $t$ such that there exists a set $T\subseteq V(G)$ with $|T|=t$ having the property that every edge is incident with a vertex in $T$.  Given a connected graph $G$ and vertices $u,v\in V(G)$, the distance between $u$ and $v$, denoted by $\mathrm{dist}(u,v)$, is the length of the shortest path between $u$ and $v$.  The \emph{radius} of a graph is defined as $\min_{u\in V(G)}\max_{v\in V(G)}\mathrm{dist}(u,v)$.  Note that a tree $T$ has radius 1 if and only if $T$ is a star.  If a tree $T$ has radius 2, then there exists $u\in V(T)$ such that for all $v\in V(T)$, the distance from $u$ to $v$ is at most 2 (and for every vertex $u\in V(T)$, there exists a vertex $v\in V(T)$ such that the distance from $u$ to $v$ is at least 2).

Given a positive integer $n$, we write $[n]=\{1,2,\dots, n\}$.  We write $\log$ to denote the natural logarithm.  We write $G(n,p)$ for the binomial random graph on $n$ vertices with edge probability $p$.  We use the standard $O(\cdot)$, $\Omega(\cdot)$, and $\Theta(\cdot)$ notation.  When the hidden constant term may depend on a variable $\ell$, we write $O_\ell(\cdot)$ for example.   

\subsection{Preliminary material}

Let $q$ be an integer with $q\geq 2$.  For our purposes, an \emph{affine plane} of order $q$ is a $q$-uniform hypergraph $A_q=(V, E)$ with $|V|=q^2$ and $|E|=q(q+1)$ having the property that for all distinct $u,v\in V$, there exists a unique $e\in E$ such that $\{u,v\}\subseteq e$ and that $E$ can be decomposed into $q+1$ many perfect matchings, called \emph{parallel classes}.  It is known that an affine plane of order $q$ exists whenever $q$ is a prime power (and it is an open problem to determine if there exists any affine plane of order $q$ where $q$ is not a prime power).  Using affine planes, it is well-known (see \cite{G11} for instance) that if $H$ is a connected graph on $n$ vertices, $r-1$ is a prime power, and $(r-1)^2$ divides $n-1$, then $R_r(H)\geq (r-1)(n-1)+1$ (where $R_r(H)$ is the ordinary $r$-color Ramsey number of $H$).  However, we would like to have a potentially weaker lower bound which holds for all $r\geq 2$ and sufficiently large $n$.  To do this, the idea is simply to use an affine plane corresponding to the largest prime power $q$ such that $q\leq r-1$.  However, to make everything quantitatively precise, we first state Bertrand's postulate (see \cite{HW} for instance) and then we describe the standard affine plane coloring suited to our particular application.  

\begin{theorem}\label{thm:bertrand}
For all integers $r\geq 3$, there exists a prime $q$ such that $\frac{r+1}{2}\leq q\leq r-1$.
\end{theorem}

\begin{lemma}[Affine plane coloring]\label{lem:affine}
For all integers $r\geq 2$ and $n\geq \frac{(r+1)^2}{2}$, if $H$ is a connected graph on $n$ vertices, then $R_r(H)\geq \frac{r}{2}n$.
\end{lemma}

\begin{proof}
When $r=2$, the statement is trivial; so suppose $r\geq 3$.  Since $r\geq 3$, there exists a prime, and thus a prime power, $q$ such that $\frac{r+1}{2}\leq q\leq r-1$ by Theorem \ref{thm:bertrand}.  Let $A_q$ be an affine plane of order $q$, with vertices $a_1, \dots, a_{q^2}$, and parallel classes $E_1, \dots, E_{q+1}$ where for all $i\in [q+1]$, $E_i$ consists of pairwise disjoint hyperedges $e^i_1, \dots, e^i_q$. 

Let $N=q^2\floor{\frac{n-1}{q}}$.  Now we use $A_q$ to describe a $(q+1)$-coloring of $K_N$ such that every monochromatic connected component has at most $n-1$ vertices.  Let $\{V_1, \dots, V_{q^2}\}$ be a partition of $V(K_N)$ such that for all $i\in [q^2]$, $|V_i|=\floor{\frac{n-1}{q}}$.  We color the edges inside the sets arbitrarily (for instance, for all $i\in [q^2]$, color every edge inside $V_i$ with color 1).  Now for distinct $i,j\in [q^2]$, we color every edge between $V_i$ and $V_j$ with color $k$ where $E_k$ is the unique parallel class which contains an edge $e^k_\ell$ such that $\{a_i, a_j\}\subseteq e^k_\ell$.  

Now for all $k\in [q+1]$, we have $q$ many pairwise disjoint connected components of color $k$, each having exactly $q\floor{\frac{n-1}{q}}\leq n-1$ vertices.  Since $q+1\leq r$, we have $R_r(H)>N\geq q^2\left(\frac{n-1-(q-1)}{q}\right)=q(n-q)\geq \frac{r}{2}n$ (where the last inequality holds since $\frac{r+1}{2}\leq q\leq r-1$ and $n\geq \frac{(r+1)^2}{2}$).
\end{proof}

The following lemma will be used a few times when determining a lower bound on the size-Ramsey number of bipartite graphs. 

\begin{lemma}\label{lem:viz}
Let $r,k\geq 1$, let $G$ be a graph, and let $X=\{v\in V(G): d(v)\leq rk-1\}$.  There exists an $r$-coloring of the edges incident with $X$ such that every vertex in $X$ is incident with at most $k$ edges of each color.  
\end{lemma}

\begin{proof}
By Vizing's theorem, we can color the edges in $G[X]$ with $rk$ many colors such that no two incident edges in $G[X]$ receive the same color.  For each vertex $v\in X$, we note that $d(v,V(G)\setminus X)\leq rk-1-d(v,X)< rk-d(v,X)$ and since there are $rk$ colors available and exactly $d(v,X)$ colors already used on edges incident with $v$, we can assign unused colors from $[rk]$ to the edges from $v$ to $V(G)\setminus X$.  Now we have a coloring of the edges incident with $X$ with $rk$ colors so that if two edges intersect in $X$, they receive different colors.  Now we partition $[rk]$ into $r$ many sets $A_1, \dots, A_r$ each of order $k$ and we recolor the edges incident with $X$ such that if an edge receives a color from the set $A_i$ we recolor it with $i$.  This gives us an $r$-coloring of the edges incident with $X$ such that every vertex in $X$ has degree at most $k$ in every color.  
\end{proof}

The following simple observation will be used often when determining a lower bound on the size-Ramsey number of bipartite graphs. 

\begin{observation}\label{obs:fit}
Let $H$ be an $(n_1,n_2,\Delta_1,\Delta_2)$-bipartite graph and let $G$ be a bipartite graph with bipartition $X,Y$.  If
\begin{enumerate}
\item $\min\{|X|, |Y|\}<\min\{n_1, n_2\}$ or $\max\{|X|, |Y|\}<\max\{n_1, n_2\}$, or
\item $\min\{\Delta(X), \Delta(Y)\}<\min\{\Delta_1, \Delta_2\}$ or $\max\{\Delta(X), \Delta(Y)\}<\max\{\Delta_1, \Delta_2\}$,
\end{enumerate}
then $H$ is not a subgraph of $G$.  
\end{observation}

\begin{proof}
Let $V_1,V_2$ be the bipartition of $H$ such that for all $i\in [2]$, $|V_i|=n_i$.  If $H$ is a subgraph of $G$, then since $H$ is connected we must have $V_i\subseteq X$ and $V_{3-i}\subseteq Y$ for some $i\in [2]$.  Thus we have $\min\{|X|, |Y|\}\geq \min\{n_1, n_2\}$ and  $\max\{|X|, |Y|\}\geq \max\{n_1, n_2\}$ and $\min\{\Delta(X), \Delta(Y)\}\geq \min\{\Delta_1, \Delta_2\}$ and $\max\{\Delta(X), \Delta(Y)\}\geq \max\{\Delta_1, \Delta_2\}$.
\end{proof}

%

The following observation gives a simple characterization of which $(n_1, n_2,\Delta_1,\Delta_2)$-bipartite graphs are stars.

\begin{observation}\label{obs:star}
Let $H$ be an $(n_1, n_2,\Delta_1,\Delta_2)$-bipartite graph.  $H$ is a star if and only if $\Delta_1=1$ or $\Delta_2=1$ or $n_1=1$ or $n_2=1$.
\end{observation}

\begin{proof}
If $H$ is a star, then clearly $n_1=1$ and $\Delta_2=1$, or $n_2=1$ and $\Delta_1=1$.  

If $n_1=1$ or $n_2=1$, then clearly $H$ is a star.  Now without loss of generality, suppose $\Delta_1=1$ and suppose for contradiction that $H$ is not a star, which by the previous sentence implies that $n_2\geq 2$.  However, since $H$ is connected, we have a path connecting distinct vertices from $V_2$ which implies that there is a vertex of degree at least 2 in $V_1$, contradicting our assumption that $\Delta_1=1$.  
\end{proof}

Beck \cite{B2} proved the following lower bound on the size-Ramsey number of trees.  While it is not stated in this way, Beck's proof actually applies to all connected bipartite graphs.  Since we will apply his result in this more general form, we give the proof below. 

\begin{proposition}\label{prop:beck}
For every connected bipartite graph $H$, $\widehat{R}(H)\geq \frac{\beta(H)}{4}$.
\end{proposition}

\begin{proof}

Let $H$ be an $(n_1, n_2,\Delta_1,\Delta_2)$-bipartite graph and set $n:=n_1+n_2$.  Suppose without loss of generality that $n_1\Delta_1\geq n_2\Delta_2$.  Let $G=(V,E)$ be a graph having the property that every 2-coloring of $G$ contains a monochromatic copy of $H$.  Let $X=\{v\in V(G): d(v)<\Delta_1\}$ and let $Y=V(G)\setminus X$.  Color all edges inside $X$ and inside $Y$ blue and all edges between $X$ and $Y$ red.  

\tbf{Case 1} ($\Delta_1\leq \Delta_2$): There can be no red copy of $H$ since every vertex in $X$ has degree less than $\Delta_1\leq \Delta_2$.  Likewise, there can be no blue copy of $H$ inside $X$.  So if there is a blue copy of $H$, it must be in $Y$, which implies $|Y|\geq n$.  Using this, together with fact that vertices in $Y$ have degree at least $\Delta_1$ and the assumption that $\Delta_1n_1\geq \Delta_2n_2$, we have  
\[e(G)\geq \frac{1}{2}|Y|\Delta_1\geq \frac{1}{2}n\Delta_1\geq \frac{1}{2}n_1\Delta_1\geq \frac{1}{4}( n_1\Delta_1+n_2\Delta_2)=\frac{\beta(H)}{4}.\]

\tbf{Case 2} ($\Delta_1>\Delta_2$):  Since every vertex in $X$ has blue degree less than $\Delta_1$, there can be no blue copy of $H$ inside $X$.  If there is a blue copy of $H$ inside $Y$, then $|Y|\geq n$ and since the vertices in $Y$ have degree at least $\Delta_1$ we have $$e(G)\geq \frac{1}{2}|Y|\Delta_1\geq \frac{1}{2}n\Delta_1>\frac{\beta(H)}{2}.$$  Finally, if there is a red copy of $H$, it must be the case that the part of size $n_1$ is embedded into $Y$ (since every vertex in $X$ has degree less than $\Delta_1$) and thus $|Y|\geq n_1$ which, together with the fact that vertices in $Y$ have degree at least $\Delta_1$ and the assumption that $\Delta_1n_1\geq \Delta_2n_2$, implies 
\[e(G)\geq \frac{1}{2}|Y|\Delta_1\geq \frac{1}{2}n_1\Delta_1\geq \frac{1}{4}(n_1\Delta_1+n_2\Delta_2)=\frac{\beta(H)}{4}.\qedhere\]
\end{proof}

We will use the following concentration inequality of McDiarmid \cite{Mc} (see \cite[Theorem 3.1]{Mc2}).  

\begin{lemma}[McDiarmid's inequality]\label{lem:mcd} Given a finite probability space, let $N$ be a positive integer, let $c_1, \dots, c_N$ be non-negative reals, let $A_1, \dots, A_n$ be subsets of $\RR$, and let $\mathbf{X}=(X_1, \dots, X_N)$ where for all $i\in [N]$, $X_i$ is a random variable with range $A_i$, and $X_1, \dots, X_N$ are mutually independent.  Let $Z:\prod_{i\in [N]}A_i\to \mathbb{R}$ such that for all $\mathbf{x}, \mathbf{x'}\in \mathbf{X}$, if $\mathbf{x}$ and $\mathbf{x'}$ differ only in coordinate $i$, then $|Z(\mathbf{x})-Z(\mathbf{x'})|\leq c_i$.  Then for all $t\geq 0$ we have
\[\mathbb{P}\sqbs{Z\ge \mathbb{E}[Z]  + t} \le \exp\Big(-\frac{2t^2}{\sum_{i\in [N]} c_i^2 }\Big).
\]
\end{lemma}

We will also use the following specific instance of Chernoff's inequality \cite{Ch} (see \cite[Theorem 2.1]{JLR}).

\begin{lemma}[Chernoff's inequality]\label{lem:chernoff}
If $X$ is a random variable with binomial distribution, then for all $t\geq 0$,  $\PP(X\geq \EE[X]+t)\leq \exp\big(-\frac{t^2}{2(\EE[X]+t/3)}\big)$.
\end{lemma}

\section{A lower bound on the size-Ramsey number of connected graphs with $m$ edges}

In this section we prove Theorem \ref{thm1}.  We split the proof into two cases depending on whether $H$ is bipartite or not.

\subsection{Non-bipartite case}

\begin{proposition}\label{prop:chi3}
For all $r\geq 2$, there exists $m_0:=m_0(r)$ such that if $H$ is a connected graph with $m\geq m_0$ edges and chromatic number at least 3, then $\widehat{R}_r(H)\geq \frac{r^2}{36} m$.
\end{proposition}

\begin{proof}
Let $G=(V,E)$ be a graph with $|E|< \frac{1}{4}r^2 m$.  We first show that for all $r\geq 2$, $G$ can be colored with $2r$ colors such that there is no monochromatic copy of $H$ (then we will apply this with $\floor{r/2}$ in place of $r$ to get the desired result).

Let $V_0:=\{v\in V(G):d(v) > r\sqrt{m}\}$. Then $\frac{1}{2}|V_0|r\sqrt{m} < |E| < \frac{1}{4}r^2 m $ implies that $|V_0| < \frac{r}{2}\sqrt{m}$.  Also note that $m\leq  \binom{|V(H)|}{2}\leq \frac{|V(H)|^2}{2}$ and thus $|V(H)|> \sqrt{2m}$.  Note that since $|V_0|<\frac{r}{2}\sqrt{m}<\frac{r}{2}|V(H)|$,  Lemma \ref{lem:affine} implies that we can color the edges inside $V_0$ with colors from $[r]$ such that there is no monochromatic copy of $H$ in $V_0$.  We use color $2r$ for all of the edges between $V_0$ and $V\setminus V_0$ and we note that since $H$ has chromatic number at least 3, there is no monochromatic copy of $H$ between $V_0$ and $V\setminus V_0$.  We now show how to color the edges inside $V\setminus V_0$ with colors from $[2r-1]$ so there is no monochromatic copy of $H$.  Note that while we are using colors from $[r]$ for edges inside $V_0$ and inside $V\setminus V_0$, there can be no monochromatic copy of $H$ which uses edges from both because $H$ is connected.  

Set $V'=V\setminus V_0$, $E'=E\cap \binom{V'}{2}$, $G'=(V',E')$, $N:=|V'|$, and let $v_1, \dots, v_N$ be an enumeration of $V'$.  Let $q$ be the smallest prime power such that $q\geq r\geq 2$ and note that by Theorem \ref{thm:bertrand} (with $r$ in place of $\frac{r+1}{2}$), we have $q\leq 2r-2$.  Assign the vertices in $V'$ independently and uniformly at random to the sets $V_1, \ldots, V_{q^2}$.  Let $L$ be a hyperedge of the affine plane $A_{q}$ on vertex set $[q^2]$ (where we recall that $|L|=q$), let $V_L=\bigcup_{i\in L}V_i$, and define the random variable $Z_L := e(G'[V_L])$.  The probability that $uv\in E'$ satisfies $uv\subseteq V_L$ is exactly $\left(\frac{|L|}{q^2}\right)^2=\frac{1}{q^2}$.  Thus we have  
\begin{equation}\label{eq:expect}
\mathbb{E}\sqbs{Z_L}=\frac{1}{q^2}|E'|< \frac{1}{q^2}\cdot\frac{1}{4}r^2 m \leq \frac{m}{4}.
\end{equation}

Note that by changing the assignment of the vertex $v_i$, we can change the value of $Z_L$ by at most $d_{G'}(v_i)$, so we will be in a position to apply Lemma \ref{lem:mcd} (McDiarmid) with $c_i=d_{G'}(v_i)$.  To say this a bit more formally, for all $i\in [N]$, let $X_i$ be a random variable which equals $\ell$ if and only if $v_i\in V_\ell$.  Note that if $\mathbf{x}$ and $\mathbf{x'}$ are two output vectors of $(X_1, \dots, X_N)$ which differ in exactly in the $i$th coordinate (that is $\mathbf{x}$ and $\mathbf{x'}$ correspond to two different assignments where the only difference is the location of $v_i$), then $|Z_L(\mathbf{x})-Z_L(\mathbf{x'})|\leq d_{G'}(v_i)$. 

In order to estimate the sum $\sum_{i\in [N]}c_i^2=\sum_{v\in V\setminus V_0}d_{G'}(v)^2$, first let $V_1=\{v\in V\setminus V_0: d_{G'}(v)> (r^2m)^{1/3}\}$ and $V_2=\{v\in V\setminus V_0: 0<d_{G'}(v)\leq (r^2m)^{1/3}\}$.  Note that $\frac{1}{2}(r^2m)^{1/3}|V_1|< |E| < \frac{1}{4}r^2 m$ and thus $|V_1|< \frac{1}{2}(r^2m)^{2/3}$.  Also we trivially have $|V_2|\leq 2|E|<\frac{1}{2}r^2m$.  So we have 
\begin{align}
\sum_{v\in V\setminus V_0}d_{G'}(v)^2=\sum_{v\in V_1}d_{G'}(v)^2+\sum_{v\in V_2}d_{G'}(v)^2 &\leq |V_1|(r\sqrt{m})^2+|V_2|((r^2m)^{1/3})^2\notag\\
&< \frac{1}{2}(r^2m)^{2/3}r^2m+\frac{1}{2}r^2m(r^2m)^{2/3}\notag\\
&=(r^2m)^{5/3} \label{eq:mcsum}
\end{align}

Now using Lemma \ref{lem:mcd} (McDiarmid) we have 
\[\mathbb{P}\sqbs{Z_L \ge m}\stackrel{\eqref{eq:expect}}{\leq} \mathbb{P}\sqbs{Z_L \ge \mathbb{E}\sqbs{Z_L}+\frac{3m}{4}} \stackrel{\eqref{eq:mcsum}}{\leq} \exp\of{-\frac{2(3m/4)^2}{(r^2m)^{5/3}}} = \exp\of{-\frac{9m^{1/3}}{8r^{10/3}}}< \frac{1}{q(q+1)},\]
where the last inequality holds since $m$ is sufficiently large in terms of $r$ and $r\leq q\leq 2r-2$.

Thus by taking a union bound over all $q(q+1)$ hyperedges of $A_{q}$, we conclude that there exists a partition of $V'$ such that at most $m - 1$ edges lie inside $V_L$ for all $L\in E(A_q)$.  Suppose $V_1, \ldots, V_{q^2}$ is such a partition.  Note that there are $q+1$ parallel classes and $q+1\leq 2r-1$.  These parallel classes will correspond to colors from $[q+1]\subseteq [2r-1]$.  For every edge $e\in E'$, we assign color $k$ to $e$ if the endpoints of $e$ are in distinct sets $V_i$ and $V_j$ where the unique hyperedge containing $i$ and $j$ in $A_{q}$ is in the $k$th parallel class of $A_{q}$.  We color $e$ with color 1 if both of its endpoints are in $V_i$ for some $i\in [q^2]$.  Note that there is no monochromatic copy of $H$ since each $V_L$ contains at most $m-1$ edges, and if $L$ and $L'$ are in the same parallel class $k$ of $A_q$, then the edges of color $k$ in $G'[V_L]$ are disconnected from the edges of color $k$ in $G'[V_{L'}]$.

Now we apply the above result with $\floor{\frac{r}{2}}$ in place of $r$ to get $\widehat{R}_r(H)\geq \frac{1}{4}\floor{\frac{r}{2}}^2 m\geq \frac{r^2}{36}m$, where we used the fact that $r\geq 2$ to get $\floor{\frac{r}{2}}\geq \frac{r}{3}$.  
\end{proof}

\subsection{Bipartite case}

Now we deal with the case where $H$ is an $(n_1, n_2,\Delta_1,\Delta_2)$-bipartite graph.  Note that the lower bound isn't explicitly written in terms of $m:=e(H)$, but as we will see in the next subsection, we can use this to derive a lower bound in terms of $m$. 

\begin{proposition}\label{prop:weakbip}
Let $H$ be an $(n_1, n_2,\Delta_1,\Delta_2)$-bipartite graph.  If $\Delta_1\geq \Delta_2\geq 2$, then for all $r\geq 2$, $\widehat{R}_r(H)\geq \frac{r^2}{36}(\Delta_2-1)(n_1+n_2)$.
\end{proposition}

\begin{proof}
Let $G=(V,E)$ be a graph with $|E|<\frac{1}{4}r^2(\Delta_2-1)(n_1+n_2)$.  We first show that for all $r\geq 2$, $G$ can be colored with at most $2r$ colors so there is no monochromatic copy of $H$ (then we will apply this with $\floor{r/2}$ in place of $r$ to get the desired result). 

Let $X=\{v\in V(G): d_G(v)\leq r(\Delta_2-1)-1\}$ and let $Y=V\setminus X$.  By Lemma \ref{lem:viz} there is an $r$-coloring of the edges inside $X$ with colors in $[r]$ such that every vertex in $X$ has degree at most $\Delta_2-1$ to $X$ in every color.  Clearly there can be no monochromatic copy of $H$ inside $X$ since $\Delta_2-1\leq \Delta_1-1$.  

Again, by Lemma \ref{lem:viz} there is an $r$-coloring of the edges between $X$ and $Y$ with colors in $[2r]\setminus [r]$ such that every vertex in $X$ has degree at most $\Delta_2-1$ to $Y$ in every color. So by Observation \ref{obs:fit} there can be no monochromatic copy of $H$ between $X$ and $Y$.

Since every vertex in $Y$ has degree at least $r(\Delta_2-1)$, we have 
$\frac{1}{2}r(\Delta_2-1)|Y|\leq |E|<\frac{1}{4}r^2(\Delta_2-1)(n_1+n_2)$ and thus $|Y|<\frac{r}{2}(n_1+n_2)\leq R_{r}(H)$, where the last inequality holds by Lemma \ref{lem:affine}.  So we may $r$-color the edges inside $Y$ with colors from $[r]$ so there is no monochromatic copy of $H$.  Note that while we have used the colors from $[r]$ for edges inside $X$ and edges inside $Y$, there can be no monochromatic copy of $H$ which uses edges from $X$ and edges from $Y$ (because $H$ is connected).

Now by applying the above result with $\floor{\frac{r}{2}}$ in place of $r$, we get $\widehat{R}_r(H)\geq \frac{\floor{\frac{r}{2}}^2}{4}(\Delta_2-1)(n_1+n_2)\geq \frac{r^2}{36}(\Delta_2-1)(n_1+n_2)$, where the last inequality holds since $r\geq 2$ and thus $\floor{\frac{r}{2}}\geq \frac{r}{3}$.
\end{proof}

\subsection{Putting the cases together}


\begin{proof}[Proof of Theorem \ref{thm1}]
Let $H$ be a connected graph with $m$ edges such that $H$ is not a star.  

If $H$ has chromatic number at least 3, then the result follows from Proposition \ref{prop:chi3}.  So suppose $H$ is an $(n_1, n_2,\Delta_1,\Delta_2)$-bipartite graph and without loss of generality $\Delta_1\geq \Delta_2$.  Let $\{V_1, V_2\}$ be the bipartition of $V(H)$ with $|V_i|=n_i$ for all $i\in [2]$.  Since $H$ is not a star, Observation \ref{obs:star} implies that $\Delta_1\geq \Delta_2\geq 2$.  Note that $\Delta_2\geq 2$ implies that $\Delta_2-1\geq \frac{\Delta_2}{2}$.  So by Proposition \ref{prop:weakbip} we have $$\widehat{R}_r(H)\geq \frac{r^2}{36}(\Delta_2-1)(n_1+n_2)\geq \frac{r^2}{72}\Delta_2(n_1+n_2)>\frac{r^2}{72}\Delta_2n_2\geq \frac{r^2}{72}m,$$ where the last inequality holds since $m=\sum_{v\in V_2}d(v)\leq \Delta_2 n_2$.  
\end{proof}

\section{A lower bound on the size-Ramsey number of connected bipartite graphs $H$ in terms of $\beta(H)$}

\begin{proof}[Proof of Theorem \ref{thm2}]  Let $H$ be an $(n_1, n_2, \Delta_1, \Delta_2)$-bipartite graph which is not a star.  Thus by Observation  \ref{obs:star}, we have $n_1, n_2, \Delta_1, \Delta_2\geq 2$.  Without loss of generality, suppose $\Delta_1n_1\geq \Delta_2n_2$.

Let $G=(V,E)$ be a graph with $|E|< \frac{1}{4}r^2(\Delta_1-1)n_1$.  We first show that for all $r\geq 2$, $G$ can be colored with at most $6r$ many colors so there is no monochromatic copy of $H$ (then we will apply this with $\floor{\frac{r}{6}}$ in place of $r$ to get the desired result).

Let $X=\{v\in V: d(v)\leq r(\Delta_1-1)-1\}$ and $Y=V\setminus X$.  By Lemma \ref{lem:viz} there is an $r$-coloring of the edges in $G[X]$ with color set $[r]$ such that every vertex in $X$ has degree at most $\Delta_1-1$ to $X$ in every color.  Thus there is no monochromatic copy of $H$ inside $X$.  

Since every vertex in $Y$ has degree at least $r(\Delta_1-1)$, we have 
$\frac{1}{2}r(\Delta_1-1)|Y|\leq |E|<\frac{r^2}{4}(\Delta_1-1)n_1,$ which implies 
\begin{equation}\label{eq:Y}
|Y|<\frac{r}{2}n_1. 
\end{equation}

By Lemma \ref{lem:affine}, we have $|Y|<\frac{r}{2}n_1\leq R_r(H)$ and thus we can color the edges inside $Y$ with colors from $[r]$ such that there is no monochromatic copy of $H$ in $Y$.  Since $H$ is connected, there can be no monochromatic copy of $H$ which uses edges from both inside $X$ and inside $Y$.  

What remains is to color the edges between $X$ and $Y$ using colors from $[6r]\setminus [r]$.

\noindent
\tbf{Case 1} ($\Delta_1\leq \Delta_2$). Since every vertex in $X$ has degree at most $r(\Delta_1-1)-1$, by Lemma \ref{lem:viz} there is an $r$-coloring of the edges between $X$ and $Y$ with color set $[2r]\setminus [r]$ such that every vertex in $X$ has degree at most $\Delta_1-1\leq \Delta_2-1$ in every color.  Thus by Observation \ref{obs:fit} there is no monochromatic copy of $H$ between $X$ and $Y$.  In this case we have colored all the edges of $G$ with colors from $[2r]$ such that there is no monochromatic copy of $H$.

\noindent
\tbf{Case 2} ($n_1\leq n_2$). Let $Y_1, \dots, Y_r$ be a partition of $Y$ into parts each of order at most $n_1-1<\min\{n_1, n_2\}$.  Now for all $i\in [r]$ color every edge from $Y_i$ to $X$ with color $r+i$ and note that by Observation \ref{obs:fit} there is no monochromatic copy of $H$ between $X$ and $Y$.  In this case we have colored all the edges of $G$ with colors from $[2r]$ such that there is no monochromatic copy of $H$.

\noindent
\tbf{Case 3} ($\Delta_1> \Delta_2$ and $n_1>n_2$).  In what follows, we split into two subcases depending on whether $\Delta_1$ is sufficiently large in terms of $n_1$.  

\tbf{Subcase 3.1} ($\Delta_1\geq \frac{\sqrt{n_1}}{4r^2}$). Assign each vertex of $Y$ independently and uniformly at random to one of the sets $Y_1, \dots, Y_{2r}$.  For all $i\in [2r]$, we have 
\begin{equation}\label{eq:E|Yi|}
\EE[|Y_i|]=\frac{1}{2r}|Y|\stackrel{\eqref{eq:Y}}{<}\frac{n_1}{4}
\end{equation}
and for all $v\in X$ and $i\in [2r]$,
\begin{equation}\label{eq:Edx}
\EE[d(x,Y_i)]\leq \frac{r(\Delta_1-1)-1}{2r}< \frac{1}{2}(\Delta_1-1).  
\end{equation}
So by Lemma \ref{lem:chernoff} (Chernoff), for all $i\in [2r]$ we have
\begin{align*}
\mathbb{P}\sqbs{|Y_i| \ge \frac{n_1}{2}}\stackrel{\eqref{eq:E|Yi|}}{\leq} \mathbb{P}\sqbs{|Y_i| \ge \EE[|Y_i|]+\frac{n_1}{4}}\stackrel{\eqref{eq:E|Yi|}}{\leq} \exp\of{-\frac{(\frac{n_1}{4})^2}{2(\frac{n_1}{4}+\frac{n_1}{12})}}=\exp\of{-\frac{3n_1}{32}}< \frac{1}{4r}
\end{align*}
and for all $v\in X$ and $i\in [2r]$, we have 
\begin{align*}
\mathbb{P}\sqbs{d(v,Y_i) \ge \Delta_1}\stackrel{\eqref{eq:Edx}}{\leq}\mathbb{P}\sqbs{d(v,Y_i) \ge \EE[d(x,Y_i)]+ \frac{\Delta_1}{2}}&\stackrel{\eqref{eq:Edx}}{<} \exp\of{-\frac{(\frac{\Delta_1}{2})^2}{2(\frac{\Delta_1}{2}+\frac{\Delta_1}{6})}}\\
&\leq \exp\of{-\frac{3\sqrt{n_1}}{64r^2}}\\
&< \frac{1}{2r^3(\Delta_1-1)n_1}< \frac{1}{4r|X|},
\end{align*}
where the last three inequalities hold since $\Delta_1\geq \frac{\sqrt{n_1}}{4r^2}$ by the case, $n_1\geq \frac{n}{2}$ is sufficiently large in terms of $r$ (and $\Delta_1\leq n_2<n_1$), and $|X|\leq 2|E|<\frac{r^2}{2}(\Delta_1-1)n_1$ respectively. 

Now by the union bound, the probability that there exists a partition $\{Y_1, \dots, Y_{2r}\}$ of $Y$ such that for all $i\in [2r]$, $|Y_i|\leq \frac{n_1}{2}$ is greater than $1/2$, and the probability that there exists a partition $\{Y_1, \dots, Y_{2r}\}$ of $Y$ such that for all $i\in [2r]$ every vertex in $X$ has degree at most $\Delta_1-1$ to $Y_i$ is greater than $1/2$. So with positive probability, there exists a partition $\{Y_1, \dots, Y_{2r}\}$ of $Y$ satisfying both.  Now  for all $i\in [2r]$ color all edges from $X$ to $Y_i$ with color $r+i$.  If there was a monochromatic copy of $H$ between $X$ and $Y_i$, then since $|Y_i|<n_1$, we must have that the part of size $n_2$ is embedded in $Y_i$ and the part of size $n_1$ is embedded in $X$, but every vertex in $X$ has degree less than $\Delta_1$ to $Y_i$, so this is impossible.  In this case we have succeeded in coloring all of the edges of $G$ with color set $[3r]$ such that there is no monochromatic copy of $H$.

\tbf{Subcase 3.2} ($2\leq \Delta_1< \frac{\sqrt{n_1}}{4r^2}$). Note that in this case we have
\begin{equation}\label{eq:n2}
\frac{n_1}{n_2}\leq \frac{n_1+n_2-1}{n_2}\leq \frac{e(H)}{n_2}\leq \Delta_2<\Delta_1<\frac{\sqrt{n_1}}{4r^2}.
\end{equation}

Let $Y_0=\{v\in Y: d(v, X)\geq \frac{r}{2}\Delta_1\frac{n_1}{n_2}\}$ and $Y'=Y\setminus Y_0$.  We have $$\frac{r}{2}\Delta_1\frac{n_1}{n_2}|Y_0|\leq e(X,Y)\leq |E|< \frac{1}{4}r^2\Delta_1n_1$$ and thus $|Y_0|< \frac{r}{2}n_2$.  As in Case 2, we can $r$-color the edges between $Y_0$ and $X$ using color set $[2r]\setminus [r]$ such that there is no monochromatic copy of $H$ (by partitioning $Y_0$ into $r$ parts $Y_1, \dots, Y_r$ each of order less than $n_2=\min\{n_1,n_2\}$ and using color $r+i$ on every edge from $Y_i$ to $X$).

It remains to deal with the edges between $X$ and $Y'$.  Let $G'=G[X,Y']$, $N_1=|X|$, $N_2=|Y'|$, $N=N_1+N_2$, and let $v_1, \dots, v_N$ be an enumeration of $V\setminus Y_0$ such that $X=\{v_1, \dots, v_{N_1}\}$ and $Y=\{v_{N_1+1}, \dots, v_{N_2}\}$.  We assign each vertex of $X$ independently and uniformly at random to one of the sets $X_1,\dots, X_{2r}$ and assign each vertex of $Y'$ independently and uniformly at random to one of the sets $Y_1', \dots, Y_{2r}'$.  For all $i,j\in [2r]$, let $Z_{i,j}=e(X_i, Y_j')$.  We have 
\begin{equation}\label{eq:Zij}
\mathbb{E}\sqbs{Z_{i,j}}=\frac{1}{(2r)^2}|E(G')|< \frac{\frac{1}{4}r^2(\Delta_1-1)n_1}{(2r)^2}< \frac{\Delta_1n_1}{16}.
\end{equation}

Note that every vertex in $X$ has degree at most $r(\Delta_1-1)-1<r\Delta_1$ to $Y'$, and by the upper bound on $\Delta_1$ in this case together with \eqref{eq:n2}, every vertex in $Y'$ has degree at most $$\frac{r}{2}\Delta_1\frac{n_1}{n_2}<\frac{r}{2}\frac{\sqrt{n_1}}{4r^2}\frac{\sqrt{n_1}}{4r^2}=\frac{n_1}{32r^3}$$ to $X$.  Thus by changing the assignment of $x\in X$, we can change the value of $Z_{i,j}$ by less than $r\Delta_1$ and by changing the assignment of $y\in Y'$, we can change the value of $Z_{i,j}$ by less than $\frac{n_1}{32r^3}$.  So we will be in a position to apply Lemma \ref{lem:mcd} (McDiarmid) with $\sum_{i\in [N]} c_i^2=\sum_{v\in X}d(v,Y')^2+\sum_{v\in Y'}d(v,X)^2$.  To say this a bit more formally, for all $k\in [N]$, let $S_k$ be a random variable which equals $\ell$ if and only if $k\in [N_1]$ and $v_k\in X_\ell$, or $k\in [N]\setminus [N_1]$ and $v_k\in Y_\ell'$.  Note that if $\mathbf{s}$ and $\mathbf{s'}$ are two output vectors of $(S_1, \dots, S_N)$ which differ in exactly in the $k$th coordinate, then $|Z_{i,j}(\mathbf{s})-Z_{i,j}(\mathbf{s'})|\leq d_{G'}(v_k)$. 

Now using the upper bound on $\Delta_1$ from this case we have
\begin{align}
\sum_{u\in X}d(u,Y')^2+\sum_{v\in Y'}d(v,X)^2=\sum_{uv\in E(X,Y')}(d_{G'}(u)+d_{G'}(v))
&<\frac{r^2}{4}\Delta_1n_1(r\Delta_1+\frac{n_1}{32r^3})\notag\\
&\leq \frac{r^2}{4}\Delta_1n_1(\frac{\sqrt{n_1}}{4r}+\frac{n_1}{32r^3})\notag\\
&\leq \frac{\Delta_1n_1^2}{64r}\label{eq:mc2}
\end{align}
where we use the fact that $n_1$ is sufficiently large in terms of $r$ to get $\frac{\sqrt{n_1}}{4r}\leq \frac{n_1}{32r^3}$ in the last inequality.

Now using Lemma \ref{lem:mcd} (McDiarmid) we have 
\begin{align*}
\mathbb{P}\sqbs{Z_{i,j} \ge \frac{\Delta_1n_1}{8}} 
\stackrel{\eqref{eq:Zij}}{\le} \mathbb{P}\sqbs{Z_{i,j} \ge \EE[Z_{i,j}]+ \frac{\Delta_1n_1}{16}}
&\stackrel{\eqref{eq:mc2}}{\le} \exp\of{-\frac{2(\frac{\Delta_1n_1}{16})^2}{ \frac{\Delta_1n_1^2}{64r}}}\\
&= \exp(-\frac{r\Delta_1}{2})\leq \exp(-r)< \frac{1}{(2r)^2},
\end{align*}
where the last two inequalities hold since $\Delta_1\geq 2$ and $r\geq 6$.

So by the union bound over the $(2r)^2$ pairs $\{X_i, Y_j'\}$ with $i,j\in [2r]$, we have a partition of $X$ into $2r$ parts $X_1,\dots, X_{2r}$ and $Y'$ into $2r$ parts $Y_1', \dots, Y_{2r}'$ such that $e(X_i, Y_j')<\frac{\Delta_1n_1}{8}\leq \frac{\beta(H)}{4}$ for all $i,j\in [2r]$.  

We are now ready to color the edges between $X$ and $Y'$.  For convenience, we let the set of colors be $[2r]\times [2]$ (but note that these $4r$ colors correspond to colors in $[6r]\setminus [2r]$).  Consider a decomposition of $K_{2r,2r}$ (with parts $\{x_1, \dots, x_{2r}\}$ and $\{y_1, \dots, y_{2r}\}$) into $2r$ perfect matchings; in other words, consider a proper $2r$-edge coloring of $K_{2r,2r}$ with colors $[2r]$. For all $i,j\in [2r]$, let $k_{i,j}\in [2r]$ be the color of the edge $\{x_i,y_j\}$.  Since the number of edges between $X_i$ and $Y_j'$ is less than $\frac{\beta(H)}{4}$, we can color the edges between $X_i$ and $Y_j'$ with two colors $(k_{i,j}, 1)$ and $(k_{i,j},2)$ such that there is no monochromatic copy of $H$ in $G[X_i, Y_j']$ by Proposition \ref{prop:beck}.  Note that if $k:=k_{i_1, j_1}=k_{i_2, j_2}$ for distinct $(i_1, j_1)$, $(i_1, j_2)$, then $i_1\neq i_2$ and $j_1\neq j_2$.  Note that we have used the same two colors $(k,1)$ and $(k,2)$ on edges in $G[X_{i_1}, Y'_{j_1}]$ and $G[X_{i_2}, Y'_{j_2}]$; however, the edges of color $(k, 1)$ and $(k,2)$ in $G[X_{i_1}, Y'_{j_1}]$ and $G[X_{i_2}, Y'_{j_2}]$ are disconnected from each other, so there is no monochromatic copy of $H$ which uses edges from both $G'[X_{i_1}, Y'_{j_1}]$ and $G'[X_{i_2}, Y'_{j_2}]$.  So we have colored the edges between $X$ and $Y'$ with $4r$ colors such that there is no monochromatic copy of $H$.  Together with the $r$ colors already used inside $X$ and inside $Y$, and the $r$ colors between $Y_0$ and $X$, we have used a total of $6r$ colors to color all of the edges of $G$ such that there is no monochromatic copy of $H$.

Finally, we apply the above result with $\floor{\frac{r}{6}}$ in place of $r$ to get \[\widehat{R}_r(H)\geq \frac{\floor{\frac{r}{6}}^2}{4}(\Delta_1-1)n_1\geq \frac{r^2}{4\cdot 144}(\Delta_1-1)n_1\geq 
\frac{r^2}{8\cdot 144}\Delta_1n_1\geq \frac{r^2}{16\cdot 144}\beta(H)=\frac{r^2}{2304}\beta(H),\]
where we used that fact that $r\geq 6$ to get $\floor{\frac{r}{6}}\geq \frac{r}{12}$, the fact that $\Delta_1\geq 2$ to get $\Delta_1-1\geq \frac{\Delta_1}{2}$, and the fact that $\Delta_1n_1\geq \Delta_2n_2$ to get $\Delta_1n_1\geq \frac{\beta(H)}{2}$.  
\end{proof}

\section{Size-Ramsey numbers of $\alpha$-full trees}

\subsection{$\alpha$-full trees}

The following lemma and corollary are implicit in \cite[Lemma 4.3]{Del}, but we include the proofs for the readers convenience.

\begin{lemma}\label{lem:bipavdeg}
Let $G$ be a bipartite graph with bipartition $\{V_1, V_2\}$ and the average degree of vertices in $V_i$ is $d_i>0$ for all $i\in [2]$. Then $G$ has a subgraph $H$ such that the minimum degree in $H$ of vertices in $V_i$ is greater than $\frac{d_i}{2}$ for all $i\in [2]$. 
\end{lemma}

Note that the average degree condition is equivalent to saying $e(G)=d_1|V_1|=d_2|V_2|$.

\begin{proof}
If there is a vertex in $V_i$ of degree at most $d_i/2$, delete it.  Repeat this process.  The total number of edges deleted is less than $\frac{d_1}{2}|V_1|+\frac{d_2}{2}|V_2|=e(G)$.  So the process must end with a non-empty subgraph which satisfies the desired conditions.
\end{proof}

\begin{corollary}\label{cor:upper}For all trees $T$ with $n_1$ vertices in one part and $n_2$ vertices in the other, $\widehat{R}_r(T)\leq (2rn_1+1)(2rn_2+1)= 4r^2n_1n_2+2r(n_1+n_2)+1$.
\end{corollary}

\begin{proof}
Let $\{V_1, V_2\}$ be the bipartition of $K:=K_{2rn_1+1, 2rn_2+1}$ so that $|V_i|=2rn_i+1$ for all $i\in [2]$.  In any $r$-coloring of $K$, the majority color class, call it $G_1$, has more than $4rn_1n_2+2n_1+2n_2$ edges, so for all $i\in [2]$, the average degree in $G_1$ of the vertices in $V_i$ is more than $2n_{3-i}$.  Now applying Lemma \ref{lem:bipavdeg} to $G_1$, we get a subgraph $H\subseteq G_1$ having the property that every vertex in $V_i$ has degree greater than $n_{3-i}$ in $H$.  Now we can greedily embed $T$ in $H$.  
\end{proof}

Now we combine Theorem \ref{thm2} and Corollary \ref{cor:upper} to determine the correct order of magnitude of the $r$-color size-Ramsey numbers of $\alpha$-full trees. 

\begin{proof}[Proof of Theorem \ref{thm4}]
Let $0<\alpha\leq 2$ and let $T$ be an $(n_1, n_2, \Delta_1, \Delta_2)$-tree such that $T$ is $\alpha$-full; i.e.~$\beta(T)\geq \alpha n_1n_2$.  By Theorem \ref{thm2} and Corollary \ref{cor:upper} we have 
\[\frac{\alpha r^2}{2304} n_1n_2\leq \frac{r^2}{2304}\beta(T)\leq \widehat{R}_r(T)\leq 4r^2n_1n_2+2r(n_1+n_2)+1.\qedhere\]
\end{proof}

\subsection{Double stars}

For double stars, we can slightly improve the lower bound implied by Theorem \ref{thm2}, Proposition \ref{prop:weakbip}, or Theorem \ref{thm4}.

\begin{proof}[Proof of Theorem \ref{thm3}] Let $r\geq 2$ and $n\geq m\geq 1$ be integers such that $n+m+2\geq \frac{(\ceiling{r/2}+1)^2}{2}$.  

The upper bound follows directly from Corollary \ref{cor:upper}.  For the lower bound, we split into two cases.

\tbf{Case 1} ($r\geq 3$).  Let $G=(V,E)$ be a graph with $|E|< \frac{\floor{r/2}\ceiling{r/2}}{4}m(n+m+2)$.  Let $X=\{v\in V: d_G(v)\leq \floor{\frac{r}{2}}m-1\}$ and let $Y=V\setminus X$.  

By Lemma \ref{lem:viz}, we can color all of the edges incident with $X$ with $\floor{\frac{r}{2}}$ colors so that every vertex in $X$ has degree at most $m$ in every color.  There is no monochromatic copy of $S_{n,m}$ incident with $X$ because the central edge would have to be adjacent to $X$, but every vertex in $X$ has degree at most $m\leq n$ (whereas the central edge of $S_{n,m}$ has a vertex of degree $m+1$ and a vertex of degree $n+1$).

Since every vertex in $Y$ has degree at least $\floor{\frac{r}{2}}m$, we have $\frac{1}{2}\floor{\frac{r}{2}}m|Y|\leq |E|<\frac{\floor{r/2}\ceiling{r/2}}{4}m(n+m+2)$ and thus  $|Y|<\frac{\ceiling{r/2}}{2}(n+m+2)\leq R_{\ceiling{r/2}}(S_{n,m})$ (where the last inequality holds by Lemma \ref{lem:affine} since $n+m+2\geq \frac{(\ceiling{r/2}+1)^2}{2}$). So there is a coloring of the edges in $G[Y]$ with the other $\ceiling{r/2}$ colors so that there is no monochromatic copy of $S_{n,m}$ in $G[Y]$.  

Thus we have $\widehat{R}_r(S_{n,m})\geq \frac{\floor{r/2}\ceiling{r/2}}{4}m(n+m+2)\geq \frac{r^2-1}{16}m(n+m+2)$.

\tbf{Case 2} ($r=2$).  Let $G=(V,E)$ be a graph with $|E|<\frac{1}{2}(m+1)(n+m+2)$.  Let $X=\{v\in V: d(v)\leq m\}$ and $Y=V\setminus X$.  Color all edges incident with $X$ red and all of the remaining edges (i.e.~the edges inside of $Y$) blue.  As before, there is no red copy of $S_{n, m}$ because the central edge must be incident with $X$, but every vertex in $X$ has degree at most $m$.  Now since every vertex in $Y$ has degree at least $m+1$ we have $\frac{1}{2}(m+1)|Y|\leq |E|< \frac{1}{2}(m+1)(n+m+2)$ and thus $|Y|<n+m+2$.  So there is no blue copy of $S_{n,m}$.
\end{proof}

Note that the lower bound in Theorem \ref{thm3} can be improved by a factor of 2 whenever there exists an affine plane of order $\ceiling{\frac{r}{2}}-1$ (see the discussion preceding Lemma \ref{lem:affine}).  

In the case $n=m$, the upper bound can be improved a bit further using the best upper bounds on the $r$-color Ramsey number of $S_{n,n}$ or the $r$-color bipartite Ramsey numbers of $S_{n,n}$ (which is defined to be the smallest integer $N$ such that in every $r$-coloring of $K_{N,N}$, there is a monochromatic copy of $S_{n,n}$).  In particular, it follows from results in \cite{BDO} that $\widehat{R}_r(S_{n,n})\leq \min\{\frac{(2r-1)^2}{2}(n+1)^2, \big(2r-3+\frac{2}{r}+O(\frac{1}{r^2})\big)^2n^2\}$.  In the very special case when $r=2$ and $n=m$, it follows from Theorem \ref{thm3} (for the lower bound) and \cite{BDO, HJ} (for the upper bound) that $(n+1)^2\leq \widehat{R}(S_{n,n})\leq (2n+1)^2.$  It would be interesting to see if the upper bounds on $\widehat{R}_r(S_{n,m})$ can be improved further by using something other than a complete or complete bipartite host graph.

\begin{problem}
For all $r\geq 2$, improve the bounds on $\widehat{R}_r(S_{n,m})$; in particular, when $n=m$.  
\end{problem}

\section{Conclusion and open problems}\label{sec:con}

Theorem \ref{thm2} combined with \eqref{eq:del} implies that for all non-star trees $T$, we have $\Omega(r^2\beta(T))=\widehat{R}_r(T)=O_r(\beta(T))$.  It would be interesting to determine a more explicit upper bound on $\widehat{R}_r(T)$.  As mentioned in the introduction, Dellamonica's upper bound on the 2-color size-Ramsey numbers of trees is actually a consequence of a stronger result which just as easily gives an upper bound on the $r$-color size-Ramsey number of trees.  What he proves is that for all $(n_1, n_2, \Delta_1, \Delta_2)$-trees $T$ and $0<\gamma\leq 1$, there exists a graph $G$ with at most $O_\gamma(\beta(T))$ edges such that every subgraph $G'\subseteq G$ with $e(G')\geq \gamma e(G)$ contains a copy of $T$.  So if one applies this result with $\gamma=\frac{1}{r}$, the upper bound follows.  While it may be possible to go through Dellamonica's paper and determine an explicit constant depending on $\gamma$ (and thus on $r$), it does not seem like a trivial matter to do so.

\begin{problem}\label{prob:final}~
\begin{enumerate}
\item Is it true that for all $r\geq 2$ and all trees $T$, we have $\widehat{R}_r(T)=O(r^2(\log r) \beta(T))$?  If so, this would match the best known upper bound for paths and bounded degree trees.  

\item Is it true that for all $r\geq 2$ and all trees $T$ we have $\widehat{R}_r(T)=O(r^2 \beta(T))$?  If so, this would match the lower bound for all non-star trees and match the upper bound for $\alpha$-full trees.
\end{enumerate}
\end{problem}

\noindent
\textbf{Note added in paper:} Shortly before the final revision of the paper,
Beke, Li, and Sahasrabudhe \cite{BLS} proved that $\widehat{R}_r(P_n)=\Theta(r^2(\log r)n)$.  This is a major breakthrough which drastically alters the picture regarding the $r$-color size-Ramsey numbers of trees.  As it relates to this paper, their result in particular implies that Problem \ref{prob:final}(ii) has a negative answer.  It also implies that the lower bounds in Theorem \ref{thm1} and Theorem \ref{thm2} can be improved by a factor of $\log r$ in certain cases (at the moment, the only such case we know of is paths, but it seems likely that their result can be generalized to bounded degree trees).  However, this state of affairs is perhaps more intriguing because it means that for certain trees $P_n$, their result implies $\widehat{R}_r(P_n)=\Theta(r^2 (\log r) \beta(P_n))$, but on the other hand, for $\alpha$-full trees $T$ (where $T$ is not a star), Theorem \ref{thm4} implies $\widehat{R}_r(T)=\Theta(r^2 \beta(T))$.  So it would be interesting to understand what properties of trees $T$ (other than $\beta(T)$) influence the $r$-color size-Ramsey number of $T$.

\subsection*{Acknowledgements and AI tools declaration} Thank you to Deepak Bal for our fruitful conversations on this topic.  Many thanks to the referees whose careful reading led to a number of improvements throughout the paper.

As part of this final revision, I essentially used Gemini as an additional referee.  During this process, Gemini discovered an error in the original version of Proposition \ref{prop:K2} (basically, the upper bound on $d'$ wasn't depending on $\Delta$ in any way) which had a major impact on the statement of Theorem \ref{thm:kriv}.  I was originally claiming that Krivelevich's upper bound on the $r$-color size Ramsey number of paths could be generalized to say that $n$-vertex trees $T$ with maximum degree $\Delta$ satisfy $\widehat{R}_r(T) = O(\Delta r^2(\log r)n)$); however, after fixing the error, the bound becomes $\widehat{R}_r(T) = O(\Delta^3 r^2\log(\Delta r)n)$.  I would like to reiterate that this error was purely my fault; that is, there is no such error in Krivelevich's paper, nor does Krivelevich make any claims regarding generalizations to bounded degree trees.

\bibliographystyle{abbrv}
\bibliography{references}

\newpage

\section{Appendix: An upper bound on the size-Ramsey numbers of bounded degree trees}\label{sec:app}

As mentioned in the introduction, Krivelevich \cite{K2} proved that $\widehat{R}_r(P_n)=O(r^2(\log r) n)$.  However, his method of proof is more general and (after an appropriate modification of the calculations) yields an explicit upper bound of the same type for all bounded degree trees. To make this result concrete, we will do the calculations in this appendix. 

We begin with a classic result of Friedman and Pippenger \cite{FP}.

\begin{theorem}[Friedman, Pippenger]\label{thm:FP}Let $\Delta$ and $n$ be positive integers and let $G$ be a graph. If for all $X\subseteq V(G)$ with $|X|\leq 2n-2$ we have $|N(X)\setminus X|\geq \Delta|X|$, then $G$ contains every tree with $n$ vertices and maximum degree at most $\Delta$. 
\end{theorem}

The following is exactly \cite[Proposition 3]{K2} and \cite[Proposition 6.2]{K1}. 

\begin{proposition}[Krivelevich]\label{prop:K1}
Let $c_1>c_2>1$ be reals and let $\delta=(\frac{c_2}{5c_1})^{\frac{c_2}{c_2-1}}$.  If $G=G(n,\frac{c_1}{n})$, then w.h.p.~every set of $k\leq \delta n$ vertices of $G$ spans fewer than $c_2k$ edges.  
\end{proposition}

The following is a slight generalization of \cite[Proposition 7.2]{K2}.

\begin{proposition}[Krivelevich]\label{prop:K2}
Let $d,d',r,\Delta$ be positive reals such that $d'\leq \frac{d}{4r(\Delta+1)}$ and let $G=(V,E)$ be a graph with average degree at least $d$.  If every subset $W\subseteq V$ with $|W|< (2\Delta+2)n$ spans fewer than $d'|W|$ edges, then for all $E'\subseteq E$ with $|E'|\geq \frac{|E|}{r}$, there exists $V'\subseteq V$ such that $G'=(V', E'\cap \binom{V'}{2})$ has the property that every set $X\subseteq V'$ with $|X|\leq 2n$ satisfies $|N_{G'}(X)\setminus X|\geq \Delta|X|$.
\end{proposition}

\begin{proof}
Let $E'\subseteq E$ with $|E'|\geq \frac{|E|}{r}$ and let $G'=(V,E')$.  Since $G'$ has average degree at least $d/r$, there exists $V'\subseteq V$ such that $G'[V']$ has minimum degree at least $\frac{d}{2r}$.  
Now let $X\subseteq V'$ with $|X|\leq 2n$.  If $|X\cup N_{G'}(X)|\geq (2\Delta+2)n$, then since $|X|\leq 2n$, we have $|N_{G'}(X)\setminus X|\geq 2\Delta n\geq \Delta|X|$ as desired; so suppose $|X\cup N_{G'}(X)|< (2\Delta+2)n$.

In this case we have
$$d'(|X|+\Delta|X|)\geq d'|X\cup N_{G'}(X)|>e_{G'}(X\cup N_{G'}(X))\geq \frac{1}{2}|X|\delta(G')\geq \frac{d}{4r}|X|,$$
where the second inequality holds by the assumption that the number of edges spanned by $X\cup N_{G'}(X)$ is less than $d'|X\cup N_{G'}(X)|$.  But this implies $d'>\frac{d}{4r(\Delta+1)}$, a contradiction. 
\end{proof}

Finally, we have the desired strengthening of \cite[Theorem 5.1]{K1} from paths to bounded degree trees. 

\begin{theorem}[Krivelevich]\label{thm:kriv}
For all $\Delta\geq 2$ and $r\geq 2$ there exists $n_0$ such that if $T$ is a tree on $n\geq n_0$ vertices with maximum degree at most $\Delta$, then $\widehat{R}_r(T)\leq 5600 \Delta^3 r^2 \log(r\Delta) n$.
\end{theorem}

\begin{proof}Let $\Delta\geq 2$ and $r\geq 2$.  Let $n_0$ be a constant whose value will be determined later but only depends on $\Delta$ and $r$.  Let $T$ be a tree on $n\geq n_0$ vertices with maximum degree at most $\Delta$.  Let $N=180r\Delta^2n$ and $p=1.01\frac{60r\Delta\log(r\Delta)}{N}$.  Set $c_1:=60r\Delta\log(r\Delta)$, $c_2:=10\log(r\Delta)$, $\delta=(\frac{c_2}{5c_1})^{\frac{c_2}{c_2-1}}$, and note that\footnote{After rearranging, this amounts to checking that $e^\frac{\log (30r\Delta)}{10\log (r\Delta)-1}\leq 2$ which can be confirmed when $r\Delta=4$ and it can easily be shown that $e^\frac{\log (30x)}{10\log x-1}$ is decreasing for all $x\geq 4$.} $60r\Delta\left(\frac{1}{30r\Delta}\right)^{1+\frac{1}{10\log (r\Delta)-1}}=2\left(\frac{1}{30r\Delta}\right)^{\frac{1}{10\log (r\Delta)-1}}\geq 1$ for all $r\Delta\geq 4$. So we have
\begin{equation}\label{eq:delta}
\delta N=180r\Delta^2 n\left(\frac{1}{30r\Delta}\right)^{1+\frac{1}{10\log (r\Delta)-1}}\geq 3\Delta n \geq (2\Delta+2)n.
\end{equation}

Note that since $n_0$ is sufficiently large we have that w.h.p.~$G=G(N,p)$ has at most $1.01p\frac{N^2}{2}\leq 1.01^2 \cdot 5400 r^2 \Delta^3 \log(r\Delta) n \leq 5600 r^2 \Delta^3 \log(r\Delta) n$ edges and average degree at least $\frac{1}{1.01}pN=60r\Delta\log(r\Delta)$.  Furthermore by \eqref{eq:delta} and Proposition \ref{prop:K1} and the fact that $n_0$ is sufficiently large we have that w.h.p.~$G=G(N,p)$ has the property that every set $X\subseteq V(G)$ with $|X|\leq (2\Delta+2)n$ spans fewer than $10\log(r\Delta)|X|$ edges.  Consider an $r$-coloring of the edges of $G$ and let $G_1$ be the subgraph induced by the edges in the majority color class (so that $e(G_1)\geq \frac{1}{r}e(G)$).  Since $\Delta \geq 2$, we have $\Delta+1 \leq \frac{3}{2}\Delta$, which implies $c_1 = 6r\Delta c_2 \geq 4r(\Delta+1)c_2$. Thus by Proposition \ref{prop:K2} (with $d=c_1$ and $d'=c_2$), we have that $G_1$ contains a subgraph $G_1'$ with the property that every set $X\subseteq V(G_1')$ with $|X|\leq 2n$ satisfies $|N_{G_1'}(X)\setminus X|\geq \Delta|X|$.  Thus by Theorem \ref{thm:FP}, $G_1'$ contains a copy of $T$.
\end{proof}

For trees $T$ of maximum degree $\Delta$ it would be very interesting to determine the exact role of $\Delta$ in the upper bound on $\widehat{R}_r(T)$.  Is it possible that $\widehat{R}_r(T)=O(\Delta r^2(\log r) n)$?  If so, then can the result of Beke, Li, and Sahasrabudhe \cite{BLS} be generalized to show $\widehat{R}_r(T)=\Theta(\Delta r^2(\log r) n)$? 
\end{document}